\documentclass[12pt,reqno,a4paper]{amsart}
\usepackage{blindtext}
\usepackage{fullpage}
\usepackage{mathtools}
\usepackage{longtable}
\usepackage{amsmath,amssymb,amsthm}
\usepackage{amscd}
\usepackage{bm}
\usepackage{hyperref}   
\usepackage{dsfont}
\usepackage{enumerate}
\usepackage{epsfig}
\usepackage{float, graphicx}
\usepackage{latexsym, amsxtra}
\usepackage{mathrsfs}
\usepackage{multicol}
\usepackage[normalem]{ulem}
\usepackage{psfrag}
\usepackage[parfill]{parskip}
\usepackage{stmaryrd}
\usepackage{tikz}
\usepackage[T1]{fontenc}
\usepackage{url}
\usepackage{verbatim}
\usepackage{indentfirst}
\usepackage{tikz-cd}
\usepackage{booktabs}

\usepackage{mathtools}

\makeatletter
\def\thm@space@setup{%
	\thm@preskip=2ex \thm@postskip=2ex
}
\makeatother

\hypersetup{hidelinks}

\newtheorem{thm}{Theorem~}[section]
\newtheorem{lem}[thm]{Lemma~}

\newtheorem{prop}[thm]{Proposition~}
\newtheorem{ques}[thm]{Question~}
\newtheorem{cor}[thm]{Corollary~}

\newtheorem{conj}[thm]{Conjecture~}

\theoremstyle{remark}

\theoremstyle{definition}  

\newcommand{\calA}{\mathcal{A}}
\newcommand{\calL}{\mathcal{L}}

\newcommand{\CC}{\mathbb{C}}
\newcommand{\ZZ}{\mathbb{Z}}
\newcommand{\RR}{\mathbb{R}}

\newcommand{\PP}{\mathbb{P}}
\newcommand{\FF}{\mathbb{F}}

\newcommand\id{\mathrm{id}}

\newcommand\mult{\mathrm{mult}}

\newcommand{\fp}{\mathfrak{p}}

\DeclareMathOperator{\Supp}{Supp}
\DeclareMathOperator{\PD}{PD}
\DeclareMathOperator{\Irr}{Irr}

\title{Monodromy Eigenvalues of Milnor Fibers for Line Arrangements}

\author[B. Xie]{Baiting Xie}
\address{Qiuzhen College, Tsinghua University, China}
\email{xbt23@mails.tsinghua.edu.cn}

\author[C. Yu]{Chenglong Yu}
\address{Center for Mathematics and Interdisciplinary Sciences, Fudan University and
Shanghai Institute for Mathematics and Interdisciplinary Sciences (SIMIS), Shanghai, China}
\email{yuchenglong@simis.cn}

\author[Z. Zheng]{Zhiwei Zheng}
\address{Tsinghua University, China}
\email{zhengzhiwei@mail.tsinghua.edu.cn}
\date{}

\begin{document}
		
\begin{abstract}
It is an important problem to know whether the monodromy on the cohomology of Milnor fibers associated to hyperplane arrangements is a combinatorial invariant. In this paper, we obtain a combinatorial  vanishing criterion for certain eigenspaces of this algebraic monodromy in line arrangements. Combining with Hirzebruch inequality, which is a consequence of Bogomolov--Miyaoka--Yau inequality, we prove that for essential complex line arrangements, the eigenvalues of the monodromy have orders at most five. This is a partial progress toward Papadima--Suciu conjecture and proves Salvetti--Serventi connectivity conjecture. For essential complexified real line arrangements, the monodromy order is improved to at most four thanks to Shnurnikov's inequality. This confirms Papadima--Suciu conjecture for real line arrangements and also Yoshinaga's sharp pair conjecture.

\end{abstract}
	
	\maketitle
     \setcounter{tocdepth}{1}

\section{Introduction}\label{section: introduction}

Let $\calA = \{l_{1},\cdots,l_{d}\}$ be a reduced projective line arrangement in $\CC\PP^2$ consisting of $d$ distinct lines. Denote by $\calL_2=\{p\mid p=l_i\cap l_j, i\neq j\}$ the set of all intersection points of $\calA$. In this paper we always assume $\calA$ is \textbf{essential}, that is, $\#\calL_2 > 1$. Fix a defining linear form $L_i \in (\CC^3)^*$ for each $l_i \in \calA$. Then the singularity
\begin{equation*}
    (\CC^3,\mathbf{0}) \longrightarrow (\CC,0),\quad x \mapsto \prod_{i=1}^{d}L_{i}(x)
\end{equation*}
induces an affine Milnor fibration over the punctured plane defined by 
\begin{equation*}
\{x\in \CC^3\mid \prod_{i=1}^{d}L_{i}(x)\neq 0\}\to \CC^\times,\quad
x\mapsto \prod_{i=1}^{d}L_{i}(x).
\end{equation*}
The Milnor fiber of $ \calA $ is then defined by $ F = \{\prod_{i=1}^{d}L_{i}(x) = 1 \} \subseteq \CC^{3}$ and the fibration gives a monodromy action on $F$
\[ h \colon F \rightarrow F,\quad x\mapsto \exp({2\pi \sqrt{-1}\over d}) x.\]
By quotienting the action of $h$, we obtain a $ d $ to $ 1 $ covering map $  f \colon F \rightarrow U $, where $U = \{\prod_{i=1}^{d}L_{i} \neq 0 \} \subseteq \CC\PP^{2} $ is exactly the complement of the arrangement $\calA$. It is straightforward to check that $F,h,f$ are independent of the choices of $L_i$.

In the theory of line arrangements, a central and widely open problem is whether the cohomology of $F$ with monodromy action $h^*$ is combinatorial invariant or not. More explicitly, the question can be stated as below.
\begin{ques}\label{ques: combinatorial determinacy}
    For fixed $\zeta \in \CC^\times$, does the dimension of the $\zeta$-eigenspace $H^1(F,\CC)_{\zeta}$ of $h^*$ depend only on the incidence relation of $\calL_2$ and $\calA$?
\end{ques}

The pioneering work of Orlik--Solomon \cite{orlik1980combinatorics} shows that the cohomology ring $H^*(U,\ZZ)$ is a combinatorial invariant, which gives an affirmative answer to Question \ref{ques: combinatorial determinacy} when $\zeta = 1$. A famous result by Papadima and Suciu \cite{papadima2017milnor} gives a complete solution to this problem when $\calL_2$ has only points of multiplicity $2$ and $3$.

Since $h^d = \id_{F}$, each eigenvalue of $h^*$ must be a primitive $m$-th root of unity $\zeta_m$ for some $m | d$. For all known examples with $H^1(F, \CC)_{\zeta_m}\neq 0$, the integer $m \leq 4$. Combining with mod-$p$ combinatorial invariants Aomoto–Betti numbers, Papadima and Suciu proposed the following conjecture in \cite{papadima2017milnor}.

\begin{conj}[{\cite[Conjecture~1.9]{papadima2017milnor}}]\label{conj: Papadima--Suciu}
    Let $ \calA $ be an essential line arrangement in $\CC\PP^2$ and $m$ be a prime power. When $m \geq 5$, $H^1(F, \CC)_{\zeta_m} = 0 $ for any primitive $m$-th root of unity $\zeta_m$.
\end{conj}

On the other hand, In \cite{salvetti2017twisted}, Salvetti and Serventi proposed the following conjecture from the perspective of characteristic varieties.

\begin{conj}\cite[Conjecture~1]{salvetti2017twisted}\label{conj: Salvetti--Serventi connectivity}
Let $\calA$ be a line arrangement in $\CC\PP^2$. Let  $\Gamma(\calA)$ be the graph defined by vertices
$l \in \calA $ and edges $ (l,l') $ if and only if
$l\cap l' $ is a double point. If $\Gamma(\calA)$ is connected after removing some vertex, then $ H^1(F,\CC)_{\zeta_m} = 0 $  for any primitive $m$-th root of unity $\zeta_m$ when $m > 1$.
\end{conj}

When $ \calA $ is a complexified real line arrangement, or equivalently, when all $L_i$ can be taken simultaneously as real linear forms in suitable coordinates, the eigenvalues of $h^*$ seem to be more strongly restricted. In this case, Yoshinaga \cite{MR3090727} discovered that, for all known examples with $ H^1(F,\CC)_{\zeta} \neq 0 $, one has $\zeta^3 = 1 $ and $ \calA $ contains a \textbf{sharp pair}. More explicitly, there exists a pair of lines $(l,l')$ in $ \calA $ such that $ \RR\PP^2 \setminus ( l \cup l') $ has a connected component containing no points in $ \calL_{2} $. Hence it is natural to ask the following question.

\begin{conj}\cite[Conjecture~1.2]{MR3666711}\label{conj: Yoshinaga's sharp pair}
Let $ \calA $ be an essential real complexified line arrangement in $ \CC\PP^{2} $. If $ \calA $ contains a sharp pair, then $ H^1(F,\CC)_{\zeta_m} = 0$ for any primitive $m$-th root of unity $\zeta_m$ when $m \neq 1,3$.
\end{conj}

Some special cases of Conjecture \ref{conj: Salvetti--Serventi connectivity} and Conjecture \ref{conj: Yoshinaga's sharp pair} have been verified. For Conjecture \ref{conj: Salvetti--Serventi connectivity}, Bailet \cite{bailet2014monodromy} proved the case where $\calA$ satifies some extra multiplicity conditions. When $\calA$ is a complexified real line arrangement, some cases with extra conditions on its real figure have been verified by Salvetti--Serventi \cite{salvetti2017twisted} and Elduque--Cueto--Maxim \cite{elduque2025TBA}. For Conjecture \ref{conj: Yoshinaga's sharp pair}, Bailet and Settepanella \cite{MR3666711} proved the case with some extra hypothese on the homology graph of $\calA$. The first author and the second author proved in \cite{xie2025homology} that this conjecture holds for even $m$. 

The main result of this paper is the following combinatorial criterion for the vanishing of $H^1(F, \CC)_{\zeta_m}$.

\begin{thm}\label{main thm: positive-semi-definite criterion}
	Let $\calA = \{l_1,\cdots,l_d\}$ be a line arrangement in $\CC\PP^2$. Let $ m > 1 $ be an integer dividing $d$. Define  
    
    \[
\Sigma_m(\calA)=\{p \in \calL_{2}\mid m \mid \mult(p),\ \mult(p) \geq 3\}.
\]
    If there exist real numbers $\lambda_1,\cdots,\lambda_d$ such that 
    \[
\sum\limits_{p \in \Sigma_{m}(\calA)}(\sum\limits_{p \in l_{i}}\lambda_{i})^{2}-(\sum\limits_{i=1}^{d}\lambda_{i})^{2} < 0,
\]
    then $ H^{1}(F,\CC)_{\zeta_m} = 0 $ for any primitive $ m $-th root of unity $\zeta_m$.
\end{thm}

Specializing to $\lambda_i=1$ and combining with Hirzebruch's inequality for line arrangements, which is a consequence of Bogomolov--Miyaoka--Yau \cite{miyaoka1977chern} \cite{yau1977calabi} inequality, we have the following bounds on $m$.
\begin{thm}[{=Theorem~\ref{thm: upper bound on permissible eigenvalues}}]
\label{main thm: upper bound on permissible eigenvalues}
	Let $ \calA $ be an essential line arrangement in $\CC\PP^2 $. Then $ H^{1}(F,\CC)_{\zeta_m} = 0 $ for any primitive $ m $-th root of unity $\zeta_m$ when $ m \geq 6 $. 
\end{thm}
This theorem partially confirm Conjecture \ref{conj: Papadima--Suciu}. Furthermore, Theorem \ref{main thm: upper bound on permissible eigenvalues} also proves the Salvetti--Serventi connectivity conjecture. 
\begin{cor}[=Corollary \ref{corollary: connected 2}]
\label{corollary: connected}
    Conjecture \ref{conj: Salvetti--Serventi connectivity} holds.
\end{cor}

For complexified real line arrangements, using Shnurnikov's Hirzebruch type inequality \cite{shnurnikov2016tk}, we eliminate $m=5$ in Theorem \ref{main thm: upper bound on permissible eigenvalues}.
\begin{thm}[{=Theorem~\ref{thm: upper bound on permissible eigenvalues when A is real}}]
\label{main thm: upper bound on permissible eigenvalues when A is real}
	Let $ \calA $ be an essential complexified real line arrangement in $\CC\PP^2$. Then $ H^{1}(F,\CC)_{\zeta_m} = 0 $ for any primitive $ m $-th root of unity $\zeta_m$ when $ m \geq 5 $.
\end{thm}

As an application of Theorem \ref{main thm: upper bound on permissible eigenvalues when A is real}, we prove Conjecure \ref{conj: Yoshinaga's sharp pair}

\begin{cor}[=Corollay~\ref{Corollary: Yoshinaga's conjecture}]
\label{corollary: Yoshinaga's conjecture intro}
  Conjecture \ref{conj: Yoshinaga's sharp pair} holds.
\end{cor}

\textbf{Acknowledgment.}
The authors thank Laren\c{t}iu Maxim for his interest and for pointing out several errors in the initial version of this paper. We also benefit a lot from Elduque--Herradón-Cueto--Maxim's earlier work on the Salvetti--Serventi connectivity conjecture and their wonderful lecture in the Summer School on Hyperplane Arrangements and Related Topics
at Tongji University in 2025. The authors thank the organizers, especially Yongqiang Liu, for his continued interest and encouragement. The authors also thank Shing-Tung Yau for his support of Qiuzhen college and drawing our attention to Hirzebruch's work on several occasions. The second author is supported by the national key research and development program of China (No. 2022YFA1007100) and NSFC 12201337. The third author is partially supported by NSFC 12301058.

\section{Equivariant smooth compactifications of Milnor fibers}\label{section: compactification of Milnor fiber}

Keeping the notations from \S\ref{section: introduction}, let $ \pi \colon Y \rightarrow \CC\PP^{2} $ be the blowup at all points $ p \in \calL_{2} $ with $\mult(p) \geq 3$. Denote by $\mathcal{B}=\pi^{-1}(\calA)$ the preimage of $\calA$. Then $Y$ is a smooth compactification of $U$ and $\mathcal{B}=Y \setminus U$ is a simple normal crossing divisor on $Y$. For each line $ l_i \in \calA $, denote by $ \widetilde{l_i} $ the strict transform of $l_i$ under $\pi$. For each $p \in \calL_2$ with $\mult(p) \geq 3$, denote by $E_p$ the exceptional divisor at $p$ and define $E_{p}^\circ = E_p \setminus (\bigcup_{p \in l_i}\widetilde{l_i})$. The goal of this section is the following theorem.

\begin{prop}\label{thm: smooth compactification of the Milnor fiber of reduced line arrangements}
There exists a smooth compactification $ S $ of $ F $ together with extended maps fitting into the commutative diagram 
   \begin{equation*}
    \begin{tikzcd}
    F \arrow[r, hook] \arrow[d, "h"] & S \arrow[d, "h"] \\
    F \arrow[r, hook] \arrow[d, "f"] & S \arrow[d, "f"] \\
    U \arrow[r, hook] & Y
\end{tikzcd}
   \end{equation*}
    satisfying the following conditions: 
	\begin{enumerate}
		\item $ Z = S \setminus F $ is a simple normal crossing divisor with smooth irreducible components. 
        \item $ f^{-1}(\mathcal{B}) = Z $. Furthermore, for each $ p \in \calL_2 $ with $\mult(p) \geq 3 $, there exists exactly one irreducible component $ C_p $ of $Z$ containing $f^{-1}(E_p^\circ)$. All other irreducible components of $Z$ are isomorphic to $\CC\PP^1$.
		\item For any irreducible component $ C $ of $ Z $, $ h(C) = C $ . Furthermore, if $ C = C_{p} $ for some $ p \in \calL_2 $ with $\mult(p) \geq 3 $, then $ (h|_{C})^{\mult(p)} = \id_{C} $.  
	\end{enumerate}
\end{prop}

\begin{proof}
With homogeneous coordinates $x_0,x_1,x_2$ of $\CC\PP^2$ fixed, We have a natural singular compactification of $F$ defined by \[\overline{F}=\{[x_0:x_1:x_2:y]\in \CC\PP^3 \mid y^d=\prod\limits_{i=1}^{d}L_i(x_0,x_1,x_2)\}.\] The covering map $f \colon F \rightarrow U$ naturally extends to a branched covering $ f \colon \overline{F} \rightarrow \CC\PP^2$. The equivariant smooth compactification $S$ is the resolution of $\overline{F}$ with three steps. First we blow up $\overline{F}$ at preimages of those intersection points $p\in \calL_2$ with $\mult(p) \geq 3$. The second step is the normalization of such a surface and denote it by $\widetilde{F}$. The last step is the minimal resolution of $\widetilde{F}$.

The resulting surface $S$ is a smooth compactification of $F$ such that the divisor $Z = S \setminus F $ is simple normal crossing. The irreducible components of $Z$ come from the following $3$ sources:
\begin{enumerate}
    \item the strict transform of the line $\{L_i=0\}$ in $\overline{F}$,
    \item the exceptional divisors at the preimage of a double point, and
    \item the exceptional divisors at the preimage of a multiple point.
\end{enumerate}

The first kind of components are birational to $\CC\PP^1$, which implies that they are isomorphic to $\CC\PP^1$. For the second kind of components, note that the local model of $\overline{F}$ near the preimage of a double point is
$$
z^d=xy.
$$
The exceptional divisor of minimal resolution is a chain of $d$ copies of $\CC\PP^1$. The cyclic group operation on $z$-variale lifts to the minimal resolution and preserves each $\CC\PP^1$.

So it suffices to consider the third kind of components. The local model reduces to 
\[
X_{d,r}\colon z^d=\prod_{i=1}^r(x-a_i y)
\]
where the $a_i$ are distinct. Let $
	\delta=\gcd(d,r), d=\delta d_1, r=\delta r_1.$ After blowup of origin on $xy$-plane and the $(0,0,0)$ on $X_{d,r}$. We obtain a surface locally defined by equation
\begin{equation}
		\label{equation: normalized blowup equation}
z^d=u^r\prod_{i=1}^r(v-a_i).
\end{equation}
Then the normalization $\widetilde{X_{d,r}}$ of this surface is the cyclic cover of $Bl_{(0,0)}\CC^2$ with simple normal crossing branching divisor. It is well-known that $\widetilde{X_{d,r}}$ has quotient singularities. For example, see \cite[Lemma 1.1]{arapura2014hodge}. The branching divisor consists of the strict transforms $\tilde{l}_i$ of $x-a_iy=0$ with multiplicity one and the exceptional divisor $E$ with multiplicity $r$. From the construction of normalizations, the preimage of $E$ is the cyclic cover of $\CC\PP^1$ defined by equation
\begin{equation*}
	C_{d,r}\colon 	\{[s: v_0: v_1]\in \CC\PP(r_1, 1,1)\mid s^\delta=\prod_{i=1}^r(v_0-a_iv_1)\}.
	\end{equation*}
The action of $\mu_d$ on $C_{d,r}$ factors through $\mu_\delta$ and acts on variable $s$.

Now we study the minimal resolution of the remaining singularities on $\widetilde{X_{d,r}}$. They are exactly the intersection points of $C_{d,r}$ with the strict transforms of lines $x-a_i y =0$. As shown in the equation \eqref{equation: normalized blowup equation}, near one of these singularities, the surface $\widetilde{X_{d,r}}$ is the normalization of one defined by
\begin{equation*}
    z^d = u^r v.
\end{equation*}
As shown in the proof of \cite[Lemma 1.1]{arapura2014hodge}, such normalization can be described as 
\[
\CC^2/\mu_{d_1}=\mathrm{Spec}(\CC[U,V]^{\mu_{d_1}})
\]
where the elements $\tau$ of cyclic group $\mu_{d_1}$ operate on $\CC^2$ by $\tau \cdot (U,V)=(\tau U, \tau^{-r_1}V)$. The normalization map is given by $u=U^{d_1}, z=U^{r_1}V, v=V^d$. The exceptional divisor of minimal resolution of quotient singularity $\CC^2/\mu_{d_1}$ is a chain of $\CC\PP^1$. See \cite[Section 2]{hirzebruch1986singularities}. The cyclic group operation on $z$-variale lifts to the minimal resolution and preserves each $\PP^1$. 
\end{proof}

\section{Proof of Theorem \ref{main thm: positive-semi-definite criterion}}\label{section: proof of main theorem}
In this section, we prove our main Theorem \ref{main thm: positive-semi-definite criterion}. Keeping the notation from \S\ref{section: compactification of Milnor fiber}, as in Proposition \ref{thm: smooth compactification of the Milnor fiber of reduced line arrangements}, we fix an equivariant smooth compactification $S$ of $F$, associate each multiple point $p \in \calL_{2}(\calA)$ with an irreducible component $C_{p}$ of $Z = S \setminus F$, and extend $h \colon F \rightarrow F$, $f \colon F \rightarrow U$ to $h \colon S \rightarrow S$, $f \colon S \rightarrow Y$, respectively. For any $ \zeta \in \CC $, denote by $ H^{1}(S,\CC)_{\zeta} $ the $\zeta$-eigenspace of $h^*$ in $H^{1}(S,\CC)$.

The following lemma allows us to compute monodromy action on $S$ instead of $F$.

\begin{lem}\label{lem: extension of the first cohomology}
	For any $ \zeta \neq 1 $, the restriction map
	\begin{equation*}
		H^{1}(S,\CC)_{\zeta} \longrightarrow H^{1}(F,\CC)_{\zeta}
	\end{equation*}
	is surjective.
\end{lem}

\begin{proof}
Consider the long exact sequence of local cohomology associated with $Z$:
	\begin{equation*}
	 H^{1}(S, \CC) \longrightarrow H^{1}(F, \CC) \longrightarrow H_Z^2(S,\CC) \longrightarrow  H^{2}(S, \CC) \longrightarrow H^{2}(F, \CC),
	\end{equation*}
which is compatible with the action of $h^{*}$. Since $Z$ is normal crossing, denoting by $\Irr(Z)$ the set consisting of all irreducible components of $Z$, we have
        \begin{equation*}
            H_Z^2(S,\CC) \simeq \bigoplus\limits_{C \in \Irr(Z)}  H^{0}(C,\CC),
        \end{equation*}
which is $h^{*}$-invariant by Proposition \ref{thm: smooth compactification of the Milnor fiber of reduced line arrangements}. So the conclusion holds. 
\end{proof}

On $S$, we have the following vanishing lemma.

\begin{lem}\label{lem: positivity leads to vanishing}
	Let $ E $ be an $\RR$-divisor on $ S $ such that $ E^{2} > 0 $. Then for any holomorphic $ 1 $-form $ \alpha \in H^{0}(S,\Omega_{S}^{1}) $, if $ \alpha|_{\Supp E} = 0 $, then $ \alpha = 0 $.
\end{lem}

\begin{proof}
     Choose a K\"ahler form $\omega \in H^2(S,\RR)$ on $S$ and define the intersection pairing on $ H^{k}(S,\CC)(k\leq 2) $ as 
     \begin{equation*}
         \langle-,-\rangle \colon H^{k}(S,\CC) \times H^{k}(S,\CC) \rightarrow \CC,\ \langle [\beta],[\gamma] \rangle = \sqrt{-1}^{2-k}\int_S \beta \wedge \overline{\gamma} \wedge \omega^{2-k}. 
     \end{equation*}
     
     Denote by $[E] \in H_{2}(S,\RR)$ the homology class defined by $E$ and $\PD[E] \in H^{1,1}(S,\RR)$ the Poincar\'e dual of $[E]$. Since $E^2 > 0$, we have $\langle\PD[E],\PD[E]\rangle = E^2 > 0$. By the Hodge-Riemann bilinear relation, we have the signature of $ \langle-,-\rangle $ on $  H^{1,1}(S,\RR) $ is $(1,\dim H^{1,1}(S,\RR)-1)$. So the intersection form on the vector space
    \begin{equation*}
        W = \{\beta \in H^{1,1}(S,\RR) \mid \langle\beta,[E]\rangle = 0 \}
    \end{equation*}
    is negative definite.
    
    Denote by $\eta = \sqrt{-1}\alpha \wedge \overline{\alpha} $. Then $\eta$ defines a class $ [\eta]\in H^{1,1}(S,\RR)$. Since $\alpha|_{\Supp E} = 0 $, we have
    \begin{equation*}
\langle[\eta],\PD[E]\rangle = \int_{E}\eta = 0.
    \end{equation*}
    So $[\eta] \in W $. Since the intersection form is negative definite on $W$ and 
    \begin{equation*}
        \langle[\eta],[\eta]\rangle = \int_{S}\eta \wedge \eta = 0,
    \end{equation*}
    we have $[\eta]=0 $. In particular, regarding $ \alpha \in H^0(S,\Omega_S^1) $ as a cohomology class in $H^{1}(S,\CC) $, we have
    \begin{equation*}
        \langle \alpha,\alpha \rangle = \sqrt{-1}\int_S \alpha \wedge \overline{\alpha} \wedge \omega = \int_S \eta \wedge \omega = \langle[\eta],[\omega]\rangle= 0.
    \end{equation*}

   By the Hodge-Riemann bilinear relation, the restriction of $ \langle-,-\rangle $ to $H^{1,0}(S,\CC) = H^0(S,\Omega_S^1) $ is positive definite. Hence we have $ \alpha = 0 $.
\end{proof}

Now we are ready to prove our main theorem.

\begin{proof}[Proof of Theorem \ref{main thm: positive-semi-definite criterion}]
    Recall that
\[
\Sigma_m(\calA)=\{p \in L_{2}(\calA)\mid m \mid \mult(p),\ \mult(p) \geq 3\}
\]
and there exist $ \lambda_{1},\cdots,\lambda_{d} \in \RR $ such that
	\begin{equation*}
		\sum\limits_{P \in \Sigma_{m}(\calA)}(\sum\limits_{P \in l_{i}}\lambda_{i})^{2}-(\sum\limits_{i=1}^{d}\lambda_{i})^{2}  < 0. 
	\end{equation*}
	
	Consider the $ \RR $-divisor
	\begin{equation*}
		D = \sum\limits_{i=1}^{d}\lambda_{i}\widetilde{l_{i}} + \sum\limits_{p \notin \Sigma_{m}(\calA)}(\sum\limits_{p \in l_{i}}\lambda_{i})E_{p}.
	\end{equation*}
	
	By definition we have
	\begin{equation*}
		D^{2} = ((\sum\limits_{i=1}^{d}\lambda_{i})\pi^{*}H - \sum\limits_{p \in \Sigma_{m}(\calA)}(\sum\limits_{p \in l_{i}}\lambda_{i})E_{p})^{2} = (\sum\limits_{i=1}^{d}\lambda_{i})^{2} - \sum\limits_{p \in \Sigma_{m}(\calA)}(\sum\limits_{p \in l_{i}}\lambda_{i})^{2} > 0. 
	\end{equation*}
	
	Choosing $ E = f^{*}D $, we have $E^{2} = d \cdot D^{2} > 0$. Furthermore, by Proposition \ref{thm: smooth compactification of the Milnor fiber of reduced line arrangements}, for any irreducible component $C$ of $E$, either $C \simeq \CC\PP^1$, or $ h^r|_{C} = \id_C $ for some integer $ r $ not divided by $m$. In both cases we have the $\zeta_m$-eigenspace $ H^0(C,\Omega_C^1)_{\zeta_m} = 0$.

	We then prove that $ H^{0}(S,\Omega_{S}^{1})_{\zeta} = 0 $. For any $ \alpha \in H^{0}(S,\Omega_{S}^{1})_{\zeta_m} $, $ \alpha|_{C} = 0 $  for any irreducible component $C$ of $E$ since $ H^0(C,\Omega_C^1)_{\zeta_m} = 0$. Then by Lemma \ref{lem: positivity leads to vanishing} we have $ \alpha = 0 $. So $ H^{0}(S,\Omega_{S}^{1})_{\zeta_m} = 0 $. 
    
    Similarly $ H^{0}(S,\Omega_{S}^{1})_{\overline{\zeta_m}} = 0 $. So
	\begin{equation*}
		H^{1}(S,\CC)_{\zeta_m} = H^{0}(S,\Omega_{S}^{1})_{\zeta_m} \oplus H^{1}(S,O_{S})_{\zeta} = H^{0}(S,\Omega_{S}^{1})_{\zeta_m} \oplus \overline{H^{0}(S,\Omega_{S}^{1})_{\overline{\zeta_m}}} = 0.
	\end{equation*}
	Hence by Lemma \ref{lem: extension of the first cohomology} we have
    \begin{equation*}
        H^{1}(F,\CC)_{\zeta_m} = \operatorname{Im}(H^{1}(S,\CC)_{\zeta_m} \rightarrow H^{1}(F,\CC))  = 0
    \end{equation*}
\end{proof}

\section{Applications}

In this section, we demonstrate several applications of Theorem \ref{main thm: positive-semi-definite criterion}, including partial progress toward the Papadima--Suciu conjecture and the verification of Conjectures \ref{conj: Salvetti--Serventi connectivity} and \ref{conj: Yoshinaga's sharp pair}.

Keeping the notations from \S\ref{section: introduction}, denote by
\begin{equation*}
    t_{r} = \#\{p \in \calL_{2} \mid \mult(p) = r\}
\end{equation*}
the number of points of multiplicity $r$ in $\calL_2$. Then we have the following theorem. 
\begin{thm}\label{thm: square sum criterion}
	Let $\calA = \{l_1,\cdots,l_d\}$ be a line arrangement in $\CC\PP^2$. Let $ m > 1 $ be an integer dividing $d$. If 
	\begin{equation}\label{eqn: square sum}
		d^{2} > \sum\limits_{m \mid r}r^{2}t_{r},
	\end{equation}
	then $ H^{1}(F,\CC)_{\zeta_m} = 0 $ for any primitive $ m $-th root of unity $\zeta_m$.
\end{thm}

\begin{proof}
    In Theorem \ref{main thm: positive-semi-definite criterion}, choose $\lambda_i=1$. Then the inequality is exactly the one above.
\end{proof}

Combining with Hirzebruch's inequality for line arrangements, we have the following theorem.

\begin{thm}[=Theorem \ref{main thm: upper bound on permissible eigenvalues}]
\label{thm: upper bound on permissible eigenvalues}
	Suppose that $ \calA $ is essential. Then $ H^{1}(F,\CC)_{\zeta_m} = 0 $ for any primitive root $ m $-th of unity $\zeta_m$ when $ m \geq 6 $.
\end{thm}

\begin{proof}
by Theorem \ref{thm: square sum criterion}, it suffices to verify that the inequality \eqref{eqn: square sum} holds.

 Without loss of generality we may assume $ m \mid d $. So we have $d \geq 6$. Since $\calA$ is essential, we have $t_{d} = 0$. When either $t_{d-1}>0$ or $t_{d-2}>0$, we have
\begin{equation*}
    \sum\limits_{m \mid r}r^{2}t_{r} \leq \sum\limits_{r \geq 6}r^{2}t_{r} \leq (d-1)^2 < d^2.
\end{equation*}

Otherwise, $t_d = t_{d-1} = t_{d-2} = 0 $. Since $d\geq6$, the strengthened Hirzebruch inequality
\cite[inequality~(9)]{hirzebruch1986singularities} applies:
\begin{equation}
 \label{equation: strengthened Hirzebruch}
 t_2+\frac34t_3
 \geq d+\sum_{r\geq5}(2r-9)t_r.
\end{equation}
Together with the pair counting formula
\[
 \binom d2=t_2+3t_3+6t_4+
 \sum_{r\geq 5}\binom r2t_r,
\]
we have
\begin{equation*}
    \frac12 d^2 > \binom d2
 \geq d+\frac94t_3+6t_4+ 11t_5+
 \sum_{r\geq 6}
\frac{r^2+3r-18}{2} \cdot t_r \geq 
\frac12\sum_{r\geq6}r^2t_r.
\end{equation*}
So in both cases the inequality \eqref{eqn: square sum} holds, which implies that the conclusion holds. 
\end{proof}

It is worth noting that the Hirzebruch inequality \cite{hirzebruch1986singularities} used here is obtained by applying Bogomolov--Miyaoka--Yau inequality \cite{miyaoka1977chern} \cite{yau1977calabi} to branched abelian covering of blowups of $\CC\PP^2$ with respect to line arrangement. The results above can also be viewed as vanishing theorems about cohomology on cyclic coverings of blowups of $\CC\PP^2$. It is interesting to see a more direct relation through the geometry of those abelian coverings, or more explicitly, whether the curvature approach of Yau \cite{yau1977calabi} can give a direct proof of this vanishing theorem.

As an application of Theorem \ref{thm: upper bound on permissible eigenvalues}, we give a complete solution of Conjecture \ref{conj: Salvetti--Serventi connectivity}. We first consider the case where $ m $ is a prime power. Under this assumption, Bailet \cite{bailet2014monodromy} has proved the case where $ \Gamma(\calA) $ is connected. The following lemma is a slight generalization of Bailet's result, obtained by the same method, using the Aomoto complex of $\calA$.

\begin{lem}\label{lem: connected graph kills aomoto}
    Let $\fp$ be a prime and $s > 1 $ be an integer. Let $ \Gamma_{\fp}(\calA) $ be the graph defined by vertices
$l_i \in \calA $ and edges $ (l_i,l_j) $ if and only if
$\mult(l_i\cap l_j) \leq 2 $ or $\fp \nmid \mult(l_i\cap l_j) $.  If $\Gamma_{\fp}(\calA)$ is connected after removing some vertex, then $ H^1(F,\CC)_{\zeta_{\fp^s}}= 0$ for any primitive $\fp^s$-th root of unity $\zeta_{\fp^s}$.
\end{lem}

\begin{proof}
    Without loss of generality, we may assume $\fp^s \mid d$. As shown in \cite[Theorem 11.3]{papadima2010spectral} , the following inequality holds:
    \begin{equation*}
        \dim H^1(F,\CC)_{\zeta_{\fp^s}} \leq \beta_{\fp}(\calA),
    \end{equation*}
    where $ \beta_{\fp}(\calA) $ is the Aomoto--Betti number of $\calA$ over the finite field $\FF_{\fp}$. More explicitly, as shown in \cite[Theorem 3.5]{falk2004line} (see also \cite[Lemma 3.1]{papadima2017milnor}), we have
    \begin{equation*}
        \beta_{\fp}(\calA) = \dim_{\FF_\fp}Z_\fp(\calA) - 1,
    \end{equation*}
    where $ Z_\fp(\calA) \subset \FF_{\fp}^{d} $ is the subspace consisting of those vectors $ (\tau_1,\cdots,\tau_d) $ satisfying that for any $p \in \calL_2 $,
    \begin{equation*}
        \begin{cases}
            \sum\limits_{p \in l_i}\tau_i = 0 & ,\ \text{if }\fp \mid \mult(p)\text{ and }\mult(p) \geq 3, \\
            \tau_i = \tau_j \text{ for any }i,j\text{ such that }p = l_i \cap l_j & ,\ \text{otherwise}.
        \end{cases}
    \end{equation*}

When $\Gamma_{\fp}(\calA)$ is connected after removing some vertex, without loss of generality we may assume this vertex is $l_d$. Then for any $ (\tau_1,\cdots,\tau_d) \in Z_\fp(\calA) $, we have $ \tau_1=\cdots=\tau_{d-1} $. Denote by $ \tau = \tau_1$. Furthermore, choose a point $ p \in \calL_2 \cap l_d $. Then we have
\begin{equation*}
    \begin{cases}
            (\mult(p)-1)\tau+\tau_d = 0 & ,\ \text{if }\fp \mid \mult(p)\text{ and }\mult(p) \geq 3, \\
            \tau_d = \tau  & ,\ \text{otherwise}.
        \end{cases}
\end{equation*}
   In both cases we have $ \tau_d = \tau $. So $\dim_{\FF_\fp} Z_{\fp}(\calA) = 1 $ and $ \beta_\fp(\calA)=0$. So $ H^1(F,\CC)_{\zeta_{\fp^s}}= 0$. 
\end{proof}

The following result is then a direct corollary of  Lemma \ref{lem: connected graph kills aomoto} and Theorem \ref{thm: upper bound on permissible eigenvalues}.

\begin{cor}[=Corollary \ref{corollary: connected}]
\label{corollary: connected 2}
    Let $\calA$ be an arrangement of $d$ distinct lines in $\CC\PP^2$. Let  $\Gamma(\calA)$ be the graph defined by vertices
$l \in \calA $ and edges $ (l,l') $ if and only if
$l\cap l' $ is a double point. If $\Gamma(\calA)$ is connected after removing some vertex, then $\dim H^1(F,\CC)= d-1$.
\end{cor}

\begin{proof}
    It suffices to prove that for any $ m \geq 2 $ and any primitive $ m $-th root of unity $\zeta_m$,
	\[
	    H^1(F,\CC)_{\zeta_m}=0.
	\]
    Note that $ \calA $ is essential since $\Gamma(\calA)$ is connected. So the claim holds when $ m \geq 6 $ by Theorem \ref{thm: upper bound on permissible eigenvalues}. When $ 2 \leq m \leq 5 $, we have $ m $ is a prime power. Then the claim holds by Lemma \ref{lem: connected graph kills aomoto}.
\end{proof}

When $\calA$ is a complexified real line arrangement, Theorem \ref{thm: upper bound on permissible eigenvalues} can be strengthened by replacing inequality \eqref{equation: strengthened Hirzebruch} with an estimation proved by Shnurnikov \cite{shnurnikov2016tk}.

\begin{thm}[=Theorem \ref{main thm: upper bound on permissible eigenvalues when A is real}]
\label{thm: upper bound on permissible eigenvalues when A is real}
	Suppose that $ \calA $ is an essential complexified real line arrangement in $\CC\PP^2$. Then $ H^{1}(F,\CC)_{\zeta} = 0 $ for any primitive $ m $-th root of unity when $ m \geq 5 $.
\end{thm}

\begin{proof}
by Theorem \ref{thm: square sum criterion}, it suffices to verify that the inequality \eqref{eqn: square sum} holds.

 Since $\calA$ is essential, we have $t_{d} = 0$.When either $t_{d-1}>0$ or $t_{d-2}>0$, we have
\begin{equation*}
    \sum\limits_{m \mid r}r^{2}t_{r} \leq \sum\limits_{r \geq 5}r^{2}t_{r} \leq (d-1)^2 < d^2.
\end{equation*}

Otherwise, $t_d = t_{d-1} = t_{d-2} = 0 $. By Shnurnikov's inequality \cite[Theorem 1(b)]{shnurnikov2016tk} we have:
\begin{equation*}
 t_2+\frac32t_3
 \geq 8+\sum_{r\geq4}(2r-\frac{15}{2})t_r.
\end{equation*}
Combining with the pair counting formula
\[
 \binom d2=t_2+3t_3+
 \sum_{r\geq 4}\binom r2t_r,
\]
we have
\begin{equation*}
    \frac12 d^2 > \binom d2
 \geq 8+\frac32t_3+ \frac{13}{2}t_4 +
 \sum_{r\geq 5}
 \frac{r^2+3r-15}{2} \cdot t_r
 \geq \frac12\sum_{r\geq5}r^2t_r.
\end{equation*}
So in both cases the inequality \eqref{eqn: square sum} holds, which implies that the conclusion holds. 
\end{proof}

The following result is a direct corollary of Theorem \ref{thm: upper bound on permissible eigenvalues when A is real} and \cite[Theorem 1.6]{xie2025homology}.

\begin{cor}[=Corollary \ref{corollary: Yoshinaga's conjecture intro}]
\label{Corollary: Yoshinaga's conjecture}
   Let $ \calA $ be an essential real complexified line arrangement in $ \CC\PP^{2} $. If $ \calA $ contains a sharp pair, then $ H^1(F,\CC)_{\zeta_m} = 0$ for any primitive $m$-th root of unity $\zeta_m$ when $m \neq 1,3$.
\end{cor}

\begin{proof}
    The claim holds when $ m \geq 5 $ by Theorem \ref{thm: upper bound on permissible eigenvalues when A is real}, and when $ 2 \mid m $ according to \cite[Theorem 1.6]{xie2025homology}. So the claim holds for all positive integers $m$ except $1,3$.
\end{proof}
     
	 \bibliographystyle{alpha}
	\bibliography{reference}
\end{document}